\documentclass[a4paper, 12pt]{amsart}

\usepackage{amsfonts}
\usepackage{amsthm}
\usepackage{verbatim}
\usepackage{amsmath, amssymb}
\usepackage{enumerate}
\usepackage{graphicx}
\usepackage{listings}
\usepackage{color}
\usepackage{listings}
\usepackage{float}
\usepackage{multicol}
\usepackage{booktabs}
\usepackage{lscape}
\usepackage{cases}
\usepackage{multirow}
\usepackage{rotating}
\usepackage{setspace}
\usepackage{caption}
\usepackage{array}
\usepackage{thmtools, thm-restate}
\usepackage[percent]{overpic}
\usepackage[noadjust]{cite}
\usepackage[update,prepend]{epstopdf}
\usepackage[colorlinks,linkcolor=blue,citecolor=blue]{hyperref}

\usepackage{pinlabel} 

\usepackage[top=2.5cm, bottom=2.5cm, left=2.5cm, right=2.5cm]{geometry}
\newtheorem{mydef}{Definition}

\newtheorem{thm}{Theorem}
\newtheorem{lemma}{Lemma}

\newcommand{\Z}{\mathbb{Z}}

\newcommand{\Q}{\mathbb{Q}}

\newsavebox{\LRmat} 
\savebox{\LRmat}{$\left(\begin{smallmatrix}2&1\\1&1\end{smallmatrix}\right)$}

\begin{document}

\title{The classification of quasi-alternating Montesinos links}

\author{Ahmad Issa}

\address{Department of Mathematics \\
         The University of Texas At Austin \\
         Austin, TX, 78712, USA}
\email{aissa@math.utexas.edu}

\begin{abstract}
  In this note, we complete the classification of quasi-alternating Montesinos links. We show that the quasi-alternating Montesinos links are precisely those identified independently by Qazaqzeh-Chbili-Qublan and Champanerkar-Ording. A consequence of our proof is that a Montesinos link $L$ is quasi-alternating if and only if its double branched cover is an L-space, and bounds both a positive definite and a negative definite $4$-manifold with vanishing first homology.
\end{abstract}

\maketitle

\section{Introduction}
Quasi-alternating links were defined by Ozsv\'{a}th-Szab\'{o} \cite[Definition 3.1]{MR2141852} as a natural generalisation of the class of alternating links.

\begin{mydef} The set $\mathcal{Q}$ of quasi-alternating links is the smallest set of links satisfying the following:
  \begin{itemize}
    \item The unknot U belongs to $\mathcal{Q}$.
    \item If $L$ is a link with a diagram containing a crossing $c$ such that
      \begin{enumerate}[(1)]
        \item both smoothings $L_0$ and $L_1$ of the link $L$ at the crossing $c$, as in Figure \ref{fig:crossing_resolution}, belong to $\mathcal{Q}$,
        \item $\mbox{det}(L_0), \mbox{det}(L_1) \ge 1$, and
        \item $\mbox{det}(L) = \mbox{det}(L_0) + \mbox{det}(L_1)$,
        \end{enumerate}
        then $L$ is in $\mathcal{Q}$. The crossing $c$ is called a quasi-alternating crossing.
  \end{itemize}
\end{mydef}

\begin{figure}[H]
  \begin{overpic}[width=230pt]{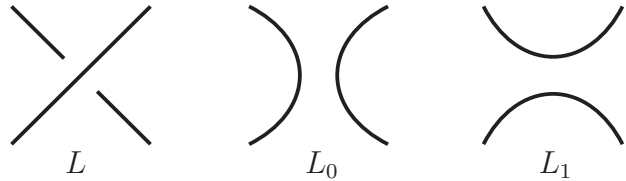}
    \put (9, 4) {$L$}
    \put (48, 4) {$L_0$}
    \put (86, 4) {$L_1$}
  \end{overpic}
\caption{$L$ and its two resolutions $L_0$ and $L_1$ in a neighbourhood of $c$.}
\label{fig:crossing_resolution}
\end{figure}

Ozsv\'{a}th-Szab\'{o} showed that the class of non-split alternating links is contained in $\mathcal{Q}$ \cite[Lemma 3.2]{MR2141852}. Moreover, quasi-alternating links share a number of properties with alternating links, we list a few of these. For a quasi-alternating link $L$:
\begin{enumerate}[(i)]
  \item $L$ is homologically thin for both Khovanov homology and knot Floer homology \cite{MR2509750}.
  \item The double branched cover $\Sigma(L)$ of $L$ is an L-space \cite[Proposition 3.3]{MR2141852}.
  \item \label{property3} The $3$-manifold $\Sigma(L)$ bounds a smooth negative definite 4-manifold $W$ with $H_1(W) = 0$ \cite[Proof of Lemma 3.6]{MR2141852}.
\end{enumerate}
For some further properties see \cite{MR3283677}, \cite{MR3361152}, \cite{MR3361803} and \cite[Remark after Proposition 5.2]{MR3071132}.

Due to their recursive definition, it is difficult in general to determine whether or not a link is quasi-alternating. For example, there still remain examples of $12$-crossing knots with unknown quasi-alternating status \cite{JablanQATable}. Champanerkar-Kofman \cite{MR2495282} showed that the quasi-alternating property is preserved by replacing a quasi-alternating crossing with an alternating rational tangle. They used this to determine an infinite family of quasi-alternating pretzel links, which Greene later showed is the complete set of quasi-alternating pretzel links \cite{MR2592726}.

Qazaqzeh-Chbili-Qublan \cite{MR3319679} and Champanerkar-Ording \cite{MR3403210} independently generalised the sufficient conditions on pretzel links to obtain an infinite family of quasi-alternating Montesinos links. This family includes all examples of quasi-alternating Montesinos links found by Widmer \cite{MR2583805}. Furthermore, it was conjectured by Qazaqzeh-Chbili-Qublan that this family is the complete set of quasi-alternating Montesinos links. We mention that Watson \cite{MR2768652} gave an iterative surgical construction for constructing all quasi-alternating Montesinos links.

Some necessary conditions to be quasi-alternating in terms of the rational parameters of a Montesinos link were obtained in \cite{MR3319679} and \cite{MR3403210} based on the fact that a quasi-alternating link is homologically thin. Further conditions are described in \cite{MR3403210} coming from the fact that the double branched cover of a quasi-alternating link is an L-space. Some additional restrictions were found in \cite{MR3361152}.

Our main result is the following theorem which states that the quasi-alternating Montesinos links are precisely those found by Qazaqzeh-Chbili-Qublan \cite{MR3319679} and Champanerkar-Ording \cite{MR3403210}:
\begin{restatable}{thm}{mainthm}
  \label{thm:necessary}
  Let $L = M(e; t_1, \ldots, t_p)$ be a Montesinos link in standard form, that is, where $t_i = \frac{\alpha_i}{\beta_i} > 1$ and $\alpha_i, \beta_i > 0$ are coprime for all $i = 1,\ldots,p$. Then $L$ is quasi-alternating if and only if
  \begin{enumerate}[(1)]
    \item $e < 1$, or
    \item $e = 1$ and $\frac{\alpha_i}{\alpha_i - \beta_i} > \frac{\alpha_j}{\beta_j}$ for some $i, j$ with $i \neq j$, or
    \item $e > p - 1$, or
      \item $e = p - 1$ and $\frac{\alpha_i}{\alpha_i - \beta_i} < \frac{\alpha_j}{\beta_j}$ for some $i, j$ with $i \neq j$.
  \end{enumerate}
\end{restatable}

As a corollary of our proof we obtain the following characterisation of the Montesinos links $L$ which are quasi-alternating in terms of their double branched covers $\Sigma(L)$:
\begin{restatable}{corol}{maincorol}
  \label{cor:qa}
  A Montesinos link $L$ is quasi-alternating if and only if
  \begin{enumerate}[(1)]
    \item $\Sigma(L)$ is an L-space, and
    \item there exist a smooth negative definite $4$-manifold $W_1$ and a smooth positive definite $4$-manifold $W_2$ with $\partial W_i = \Sigma(L)$ and $H_1(W_i) = 0$ for $i = 1,2$.
  \end{enumerate}
\end{restatable}
Note that in Corollary \ref{cor:qa} and throughout, we assume all homology groups have $\Z$ coefficients.

In light of this corollary, Theorem \ref{thm:necessary} can also be seen as a classification of the L-space Seifert fibered spaces over $S^2$ which bound both positive and negative definite $4$-manifolds with vanishing first homology. To what extent Corollary \ref{cor:qa} generalises to non-Montesinos links remains an interesting question.

This work also gives a classification of the Seifert fibered space formal L-spaces. The notion of a formal L-space was defined by Greene and Levine \cite{MR3584256} as a 3-manifold analogue of quasi-alternating links. In fact, the double branched cover of a quasi-alternating link is an example of a formal L-space. In \cite{MR3604917}, Lidman and Sivek classified the quasi-alternating links of determinant at most $7$. In fact, they show that the formal L-spaces $M^3$ with $|H_1(M)| \le 7$ are precisely the double branched covers of quasi-alternating links with determinant at most $7$. In this same direction, as a consequence of Corollary \ref{cor:qa}, we have the following.

\begin{restatable}{corol}{formallspace} A Seifert fibered space over $S^2$ is a formal L-space if and only if it is the double branched cover of a quasi-alternating link.
\end{restatable}

Corollary \ref{cor:qa} also seems significant given the recent independent characterisations of alternating knots by Greene \cite{MR3694566} and Howie \cite{MR3654110}. A non-split link is alternating if and only if it bounds negative definite and positive definite spanning surfaces (which are the checkerboard surfaces). The double branched cover of $B^4$ over such a surface is a definite $4$-manifold of the appropriate sign. Generalising this, a quasi-alternating link has the property that it bounds a pair of surfaces in $B^4$ with double branched covers a positive definite and a negative definite $4$-manifold (these surfaces cannot be embedded in $S^3$ in general). Corollary \ref{cor:qa} shows that among Montesinos links with double branched covers which are L-spaces, this property characterises those which are quasi-alternating.

Our approach to proving Theorem \ref{thm:necessary} follows that of Greene \cite{MR2592726} on the determination of quasi-alternating pretzel links. One of Greene's main strategies is as follows. Suppose $L$ is a quasi-alternating Montesinos link such that $\Sigma(L)$ is the oriented boundary of the standard negative definite plumbing $X^4$. Since the property of being quasi-alternating is closed under reflection, by property (\ref{property3}) above, $-\Sigma(L) = \Sigma(\overline{L})$ bounds a negative definite $4$-manifold $W$ with $H_1(W) = 0$. By Donaldson's theorem \cite{MR910015}, the smooth closed negative definite $4$-manifold $X \cup W$ has diagonalisable intersection form. Hence, $H_2(X)/\mbox{Tors} \hookrightarrow H_2(X \cup W)/\mbox{Tors}$ is an embedding of the intersection lattice of $X$ into the standard negative diagonal lattice. Moreover, using that $H_1(W)$ is torsion free, it is shown that if $A$ is a matrix representing the lattice embedding then $A^T$ must be surjective.

When $L$ is a pretzel link of a certain form, Greene analyses the possible embeddings of the intersection lattice of $X$ into a negative diagonal lattice and shows that the aforementioned surjectivity condition cannot hold, and hence the link cannot be quasi-alternating. Our main contribution is to argue for more general Montesinos links $L$ that there is no lattice embedding for which $A^T$ is surjective. Key to our argument are some results on lattice embeddings by Lecuona-Lisca \cite{MR2782538}. The condition we obtain combined with an obstruction based on $\Sigma(L)$ being an L-space leads to the precise necessary conditions to complete the determination of quasi-alternating Montesinos links.

\section{Preliminaries}\label{sec:preliminaries}
We briefly recall some material on Montesinos links and plumbings. See \cite{MR3403210} or \cite{MR3156509} for further detail on Montesinos links, and \cite{MR518415} for more on plumbings. The Montesinos link $M(e; t_1, \ldots, t_p)$, where $t_i = \frac{\alpha_i}{\beta_i} \in \mathbb{Q}$ with $\alpha_i > 1$ and $\beta_i$ coprime integers, and $e$ is an integer, is given by the diagram in Figure \ref{fig:montesinos}. In the figure, each box labelled $t_i$ represents the corresponding rational tangle. The $0$ rational tangle is shown in Figure \ref{fig:rational_tangle}. Introducing an additional positive (resp. negative) half-twist to the bottom of an $a/b$ rational tangle produces a rational tangle represented by $a/b + 1$ (resp. $a/b - 1$), see Figure \ref{fig:rational_tangle}. Rotating (in either direction) a rational tangle represented by $t \in \mathbb{Q} \cup \{1/0\}$ by $90$ degrees produces the rational tangle represented by $-1/t$. The rational tangle represented by any $a/b \in \mathbb{Q} \cup \{1/0\}$ can be obtained from the $0$ rational tangle by a sequence of these two operations. See \cite{MR2107964} for a more thorough treatment of rational links. Note however that an $a/b$ rational tangle with our conventions corresponds to a $b/a$ rational tangle in \cite{MR2107964}.

We also note that with our conventions for a Montesinos link $M(e; t_1, \ldots, t_p)$, the integer $e$ has opposite sign to that used by Champanerkar-Ording \cite{MR3403210}, and agrees with that of Qazaqzeh-Chbili-Qublan \cite{MR3319679} and Greene \cite{MR2592726}.

\begin{figure}[H]
  \begin{overpic}[height=150pt]{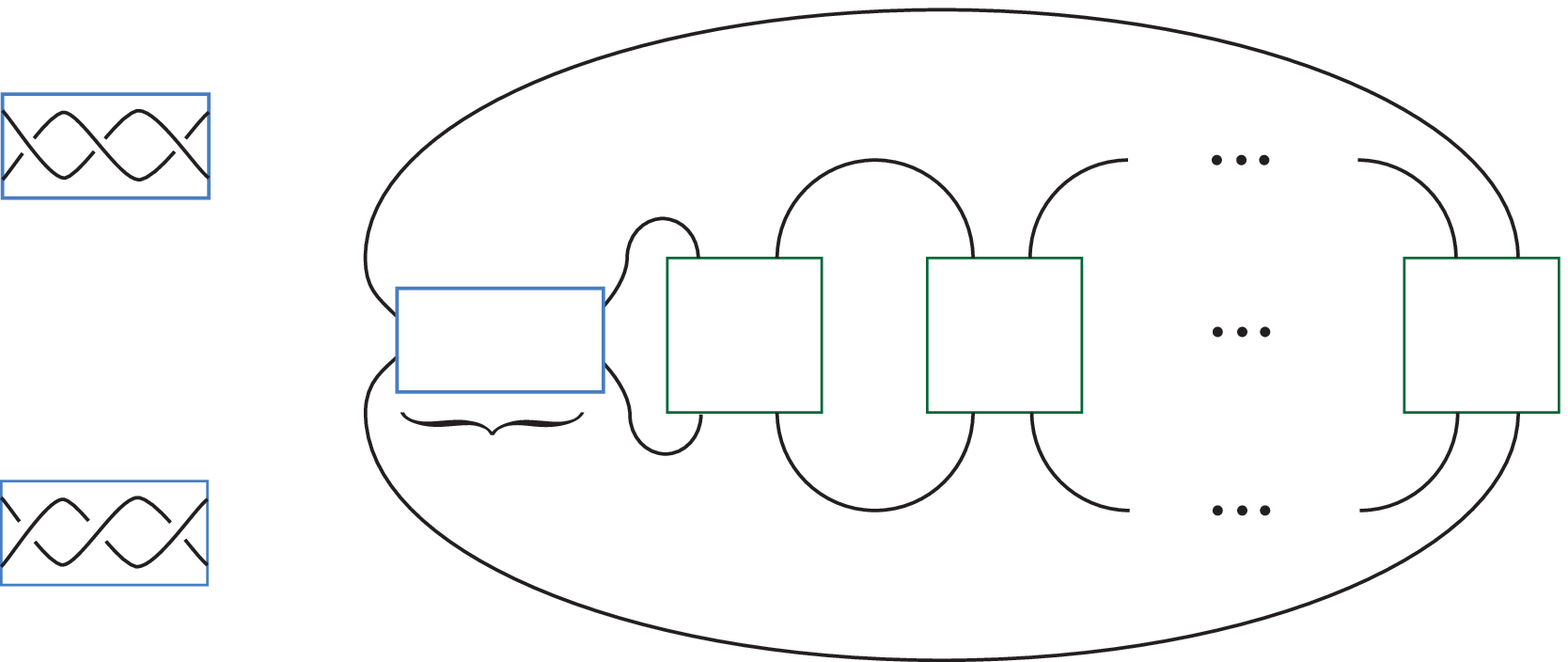}
    \put (26, 11) {\small $|e|$ crossings}
    \put (-4, 14) {\small $e < 0$ $(e=-3)$}
    \put (-2, 26) {\small $e < 0$ $(e=3)$}
    \put (46.5, 20) {\large $t_1$}
    \put (63, 20) {\large $t_2$}
    \put (93.5, 20) {\large $t_p$}
  \end{overpic}
\caption{The Montesinos link $M(e; t_1, \ldots, t_p)$ where a box labelled $t_i$ represents a rational tangle corresponding to $t_i$. The crossing type of the $|e|$ crossings depends on the sign of $e$, with the two possibilities shown on the left.}
\label{fig:montesinos}
\end{figure}

\begin{figure}[H]
  \begin{overpic}[height=120pt]{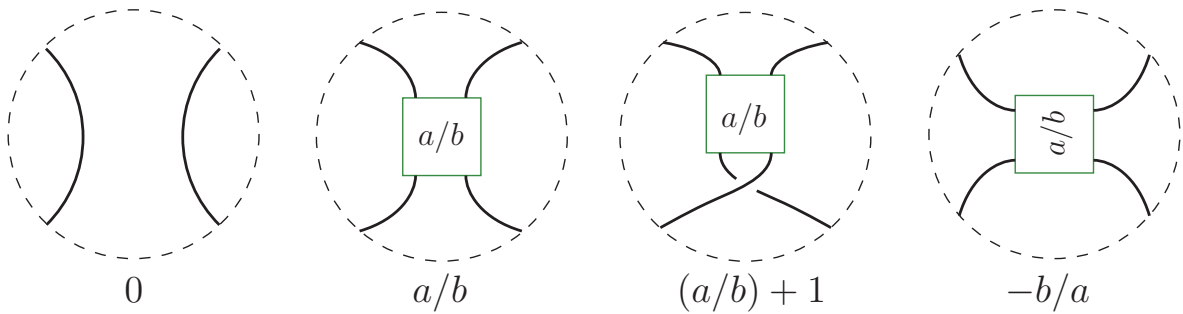}
    \put (10, 2) {\large $0$}
    \put (34.5, 2) {\large $a/b$}
    \put (35, 15.5) {$a/b$}
    \put (57, 2) {\large $(a/b) + 1$}
    \put (60.8, 17) {$a/b$}
    \put (85, 2) {\large $-b/a$}
    \put (88, 14.5) {\rotatebox{90}{$a/b$}}
  \end{overpic}
\caption{From left to right: the $0$ rational tangle, an abstract representation of a $a/b$ rational tangle, the $\frac{a}{b} + 1$ rational tangle, and the $-b/a$ rational tangle.}
\label{fig:rational_tangle}
\end{figure}

The Montesinos link $M(e; t_1, \ldots, t_p)$ is isotopic to $M(e + 1; t_1, \ldots, t_{i-1}, t_i', t_{i+1}, \ldots, t_p)$ where $t_i' = \frac{\alpha_i}{\beta_i + \alpha_i}$, and is also isotopic to $M(e - 1; t_1, \ldots, t_{i-1}, t_i', t_{i+1}, \ldots, t_p)$, where $t_i' = \frac{\alpha_i}{\beta_i - \alpha_i}$. Hence, a Montesinos link is isotopic to one in \emph{standard form}, that is, of the form $M(e; t_1, \ldots, t_p)$ where $t_i > 1$ for all $i$.

Let $L = M(e; t_1, \ldots, t_p)$ where $t_i < -1$ for all $i$. Note that any Montesinos link can be put into this form. For each $i$, there is a unique continued fraction expansion $$t_i = [a_1^i, \ldots, a_{h_i}^i] := a_1^i - \cfrac{1}{a_2^i - \cfrac{1}{\begin{aligned}\ddots \,\,\, & \\[-3ex] & a_{h_i-1}^i - \cfrac{1}{a_{h_i}^i} \end{aligned}}},$$
where $h_i \ge 1$ and $a_j^i \le -2$ for all $j \in \{1, \ldots, h_i\}$.

\begin{figure}[H]
  \begin{overpic}[height=150pt]{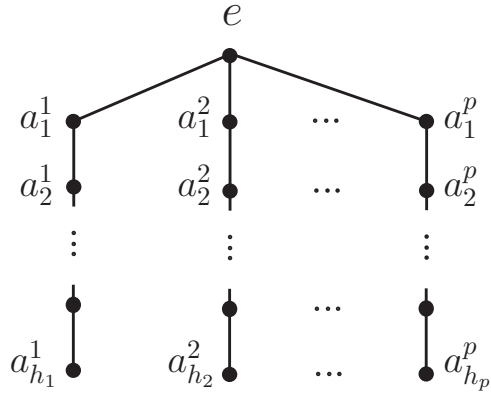}
    \put (45, 73) {\Large $e$}
    \put (4, 52) {\large $a_1^1$}
    \put (4, 38) {\large $a_2^1$}
    \put (2, 2) {\large $a_{h_1}^1$}
    
    \put (36, 52) {\large $a_1^2$}
    \put (36, 38) {\large $a_2^2$}
    \put (34, 2) {\large $a_{h_2}^2$}
    
    \put (90, 52) {\large $a_1^p$}
    \put (90, 38) {\large $a_2^p$}
    \put (90, 2) {\large $a_{h_p}^p$}    


  \end{overpic}
  \caption{The weighted star-shaped plumbing graph $\Gamma$.}
  \label{fig:plumbing}
\end{figure}

The double branched cover $\Sigma(L)$ of $L$ is the oriented boundary of the $4$-dimensional plumbing $X_\Gamma$ of $D^2$-bundles over $S^2$ described by the weighted star-shaped graph $\Gamma$ shown in Figure \ref{fig:plumbing}. We call $\Gamma$ the standard star-shaped plumbing graph for $L$. The $i$th leg of $\Gamma$ corresponding to $t_i$ is the linear subgraph generated by the vertices labelled with weights $a_1^i, \ldots, a_{h_i}^i$. The degree $p$ vertex labelled with weight $e$ is called the central vertex. Denote the vertices of $\Gamma$ by $v_1, v_2, \ldots, v_k$. The zero-sections of the $D^2$-bundles over $S^2$ corresponding to each of $v_1, \ldots, v_k$ in the plumbing together form a natural spherical basis for $H_2(X_\Gamma)$. With respect to this basis, the intersection form of $X_\Gamma$ is given by the weighted adjacency matrix $Q_\Gamma$ with entries $Q_{ij}$, $1 \le i,j \le k$ given by

$$Q_{ij} = \begin{cases} 
      \text{w}(v_i), & \mbox{if }i = j \\
      1, & \mbox{if }v_i\mbox{ and }v_j\mbox{ are connected by an edge} \\
      0, & \mbox{otherwise} 
   \end{cases},
$$
where $\text{w}(v_i)$ is the weight of vertex $v_i$. We call $(\Z^k, Q_\Gamma)$ the intersection lattice of $X_{\Gamma}$ (or of $\Gamma$).

\section{Results}
Equivalent sufficient conditions for a Montesinos link to be quasi-alternating were given in \cite[Theorem 5.3]{MR3403210} and \cite[Theorem 3.5]{MR3319679}. The goal of this section is to prove Theorem \ref{thm:necessary} which states that these sufficient conditions for a Montesinos link to be quasi-alternating are also necessary conditions.

\begin{lemma}\label{lemma:laufer} Let $L = M(e; t_1, \ldots, t_p)$, $p \ge 3$, be a Montesinos link in standard form, i.e. where $t_i = \frac{\alpha_i}{\beta_i} > 1$ and $\alpha_i, \beta_i > 0$ are coprime for all $i$. Suppose that $e \le p - 2$ and $e - \sum_{i=1}^p \frac{1}{t_i} > 0$ (in particular $e \ge 1$). Then $\Sigma(L)$ is not an L-space, and therefore $L$ is not quasi-alternating.
\end{lemma}

\begin{proof} The reflection of $L$ is given by $\overline{L} = M(e'; t_1', \ldots, t_p') = M(-e; -t_1, \ldots, -t_p)$. The space $\Sigma(\overline{L})$ is the oriented boundary of a plumbing $X_\Gamma$ corresponding to the standard star-shaped plumbing graph $\Gamma$ for $\overline{L}$. Since $e' - \sum_{i=1}^p \frac{1}{t_i'} = -\left(e - \sum_{i=1}^p \frac{1}{t_i}\right) < 0$, by \cite[Theorem 5.2]{MR518415}, $X_\Gamma$ has negative definite intersection form.

  Since $X_{\Gamma}$ is negative definite and $\Gamma$ is almost-rational, by \cite[Theorem 6.3]{MR2140997} we have that $\Sigma(\overline{L})$ is an L-space if and only if $X_\Gamma$ is a rational surface singularity (more generally, see \cite{Nemethi}). Note that $\Gamma$ is almost-rational since by sufficiently decreasing the weight of the central vertex we obtain a plumbing graph satisfying $-\text{w}(v) \ge \text{deg}(v)$ for all vertices $v$, where $\text{w}(v)$ denotes the weight of $v$, and such a graph is rational (for details see \cite[Example 8.2(3)]{MR2140997}).

  Laufer's algorithm \cite[Section 4]{MR0330500} can be used to determine whether the negative definite plumbing $X_\Gamma$ is a rational surface singularity as follows. Let $v_1, \ldots, v_k$ be the vertices of $\Gamma$ and for $i \in \{1,\ldots,k\}$, let $[\Sigma_{v_i}] \in H_2(X_\Gamma)$ be the spherical class naturally associated to $v_i$. The algorithm is as follows (see \cite[Section 3]{MR2425720} for a similar formulation).
  \begin{enumerate}
    \item Let $K_0 = \sum_{i=1}^k [\Sigma_{v_i}] \in H_2(X_\Gamma)$.
    \item In the $i$th step, consider the pairings $\langle PD[K_i], [\Sigma_{v_j}] \rangle$, for $j \in \{1,\ldots,k\}$. Note that these pairings may be evaluated using the adjacency matrix $Q$. If for some $j$ the pairing is at least $2$ then the algorithm stops and $X_\Gamma$ is not a rational surface singularity. If for some $j$, the pairing is equal to $1$, then set $K_{i+1} = K_i + [\Sigma_{v_j}]$ and go to the next step. Otherwise all pairings are non-positive, the algorithm stops and $X_\Gamma$ is a rational surface singularity.
  \end{enumerate}

  Applying Laufer's algorithm to $X_\Gamma$, we claim that the algorithm terminates at the $0^{\mbox{th}}$ step. To see this, note that for $v$ the central vertex of $\Gamma$, $\langle PD[K_0], [\Sigma_{v}]\rangle = p - e$ (each vertex adjacent to $v$ contributes $1$, the central vertex contributes $-e$). By assumption $e \le p - 2$ so $\langle PD[K_0], [\Sigma_{v}]\rangle = p - e \ge 2$. Hence, the algorithm terminates, we conclude that $X_{\Gamma}$ is not a rational surface singularity and hence $\Sigma(\overline{L})$ is not an L-space. Therefore $\Sigma(L)$ is not an L-space.
\end{proof}

The following lemma will provide an obstruction to a Montesinos link being quasi-alternating.
\begin{lemma}[{\cite[Lemma 2.1]{MR2592726}}]\label{lemma:greene_obstruction}
  Suppose that $X$ and $W$ are a pair of 4-manifolds, $\partial X = -\partial W = Y$ is a rational homology sphere, and $H_1(W)$ is torsion-free. Express the map $H_2(X)/\mbox{Tors} \rightarrow H_2(X \cup W)/\mbox{Tors}$ with respect to a pair of bases by the matrix $A$. This map is an inclusion, and $A^T$ is surjective. In particular, if some $k$ rows of $A$ contain all the non-zero entries of some $k$ of its columns, then the induced $k \times k$ minor has determinant $\pm 1$.
\end{lemma}

The following two technical lemmas will be useful when we apply the obstruction to being quasi-alternating based on Lemma \ref{lemma:greene_obstruction}.
\begin{lemma}[{\cite[Lemma 3.1]{MR2782538}}]\label{lemma:rigidity} Suppose $-1/r = [a_1, \ldots, a_n]$ and $-1/s = [b_1, \ldots, b_m]$ where $r + s = 1$. Consider a weighted linear graph $\Psi$ having two connected components, $\Psi_1$ and $\Psi_2$, where $\Psi_1$ consists of $n$ vertices $v_1, \ldots, v_n$ with weights $a_1, \ldots, a_n$ and $\Psi_2$ of $m$ vertices $w_1, \ldots, w_m$ with weights $b_1, \ldots, b_m$. Moreover, suppose that there is an embedding of the lattice $(\Z^{n+m}, Q_\Psi)$ into $(\Z^k, -\mbox{Id})$, with basis $e_1, \ldots, e_k$. For $S$ a subset of vertices of $\Psi$, define
  $$U_S = \{e_i\, |\, e_i \cdot v \neq 0\mbox{ for some }v \in S\}.$$
Suppose further that $e_1 \in U_{v_1} \cap U_{w_1}$ and $U_\Psi = \{e_1, \ldots, e_k\}$. Then $U_{\Psi_1} = U_{\Psi_2}$ and $k = n + m$.
\end{lemma}

\begin{lemma}[{\cite[Lemma 3.2]{MR2782538}}]\label{lemma:shorten_legs}
  Let $-1/r = [a_1, \ldots, a_n]$ and $-1/s = [b_1, \ldots, b_m]$ be such that $r + s \ge 1.$ Then there exists $n_0 \le n$ and $m_0 \le m$ such that $-1/r_0 = [a_1, \ldots, a_{n_0}]$ and $-1/s_0 = [b_1, \ldots, b_{m_0}]$ satisfy $r_0 + s_0 = 1$.
\end{lemma}

\mainthm*

\begin{proof} If one of the conditions (1)-(4) is satisfied then $L$ is quasi-alternating by either of \cite[Theorem 5.3]{MR3403210} or \cite[Theorem 3.5]{MR3319679}, thus it suffices to show that if none of the conditions are satisfied then $L$ is not quasi-alternating. Thus, assume none of the conditions are satisfied, in particular $p \ge 2$.

  By \cite[Section 1.2.3]{MR1941324} (see also \cite[Proposition 4.1]{MR3403210}), we have that $$\mbox{det}(L) = \left|\alpha_1 \ldots \alpha_p \left(e - \sum_{i=1}^p \frac{\beta_i}{\alpha_i} \right)\right|.$$ If $p = 2$, since none of the conditions are satisfied we must have $e = 1$ and $\frac{\alpha_1}{\alpha_1 - \beta_1} = \frac{\alpha_2}{\beta_2}$. Hence, $\mbox{det}(L) = |\alpha_1 \alpha_2 (1 - \frac{\beta_1}{\alpha_1} - \frac{\beta_2}{\alpha_2})| = 0$,
  and so $L$ is not quasi-alternating (in fact $L$ must be the two component unlink). For the remainder of the argument we assume that $p \ge 3$, and $\mbox{det}(L) \neq 0$, that is, $e - \sum_{i=1}^p \frac{\beta_i}{\alpha_i} \neq 0$.

  First consider the case $1 < e < p - 1$. The reflection of $L$ is given by
  $$\overline{L} = M \left(-e, -\frac{\alpha_1}{\beta_1}, \ldots, -\frac{\alpha_p}{\beta_p} \right) = M \left(p - e, \frac{\alpha_1}{\alpha_1 - \beta_1}, \ldots, \frac{\alpha_p}{\alpha_p - \beta_p}\right),$$
  where the latter is written in standard form and $1 < p - e < p - 1$. Moreover, we see that a reflection reverses the sign of $e - \sum_{i=1}^p \frac{\beta_i}{\alpha_i}$ and thus by a reflection if necessary we may assume that $e - \sum_{i=1}^p \frac{\beta_i}{\alpha_i} > 0$. Then by Lemma \ref{lemma:laufer}, $\Sigma(L)$ is not an L-space, so $L$ is not quasi-alternating.

  It remains to consider the cases $e = 1$ and $e = p - 1$. By a reflection if necessary we may assume that $e = 1$. Note that conditions (2) and (4) are equivalent under a reflection. We assume that condition (2) is not satisfied. We need to prove that this implies that $L$ is not quasi-alternating. If $e - \sum_{i=1}^p \frac{\beta_i}{\alpha_i} > 0$ then by Lemma \ref{lemma:laufer}, $\Sigma(L)$ is not an L-space, and therefore $L$ is not quasi-alternating.

Otherwise $e - \sum_{i=1}^p \frac{\beta_i}{\alpha_i} < 0$. We have that $$L = M\left(1; \frac{\alpha_1}{\beta_1}, \ldots, \frac{\alpha_p}{\beta_p}\right) = M \left(1 - p; \frac{\alpha_1}{\beta_1 - \alpha_1}, \ldots, \frac{\alpha_p}{\beta_p - \alpha_p}\right),$$
where $\frac{\alpha_i}{\beta_i - \alpha_i} < -1$ for all $i$.

The double branched cover $\Sigma(L)$ of $L$ is therefore the boundary of a plumbing $4$-manifold $X_{\Gamma}$ on the standard star-shaped planar graph $\Gamma$ with central vertex of weight $-(p - 1)$ and legs corresponding to the fractions $\frac{\alpha_i}{\beta_i - \alpha_i}$, $i \in \{1,\ldots,p\}$.  Our assumption that $e - \sum_{i=1}^p \frac{\beta_i}{\alpha_i} < 0$ implies that $X_{\Gamma}$ is negative definite \cite[Theorem 5.2]{MR518415}. Suppose for the sake of contradiction that $L$ is quasi-alternating. Then $\overline{L}$ is quasi-alternating and $-\Sigma(L) = \Sigma(\overline{L})$ bounds a negative definite $4$-manifold $W$ with $H_1(W) = 0$ \cite[Proof of Lemma 3.6]{MR2141852}. By Donaldson's theorem \cite{MR910015}, the smooth closed negative definite $4$-manifold $X_{\Gamma} \cup W$ has diagonalisable intersection form. Thus, the map $H_2(X_{\Gamma})/\mbox{Tors} \hookrightarrow H_2(X_{\Gamma} \cup W)/\mbox{Tors}$ induced by the inclusion map is an embedding of the intersection lattice $(\Z^k, Q_\Gamma)$ of $X_{\Gamma}$ into the standard negative diagonal lattice $(\Z^n, -\mbox{Id})$ for some $n$. Denote by $e_1, \ldots, e_n$ a basis for $(\Z^n , -\mbox{Id})$.

We use the lattice embedding to identify elements of $(\Z^k, Q_\Gamma)$ with their image in $(\Z^n, -\mbox{Id})$. For convenience, we will not distinguish between a vertex of $\Gamma$ and the vector it corresponds to in the lattice. The central vertex $v$ of $\Gamma$ has weight $-(p - 1)$, and so $v \cdot e_i \neq 0$ for at most $p-1$ values of $i \in \{1, \ldots, n\}$. Thus, by applying an automorphism if necessary, we may assume that $v$ pairs non-trivially with precisely $e_1, \ldots, e_m$ where $m \le p - 1$. Since there are $p$ legs, by the pigeonhole principle there must exist some $e_j$, where $j \in \{1,\ldots, m\}$, and two distinct vertices $v_1, v_2$ adjacent to $v$ with $v_1 \cdot e_j \neq 0$ and $v_2 \cdot e_j \neq 0$. Without loss of generality we assume that $j = 1$ and that for $i \in \{1,2\}$, the vertex $v_i$ belongs to the $i$th leg of $\Gamma$, i.e. corresponding to the fraction $\frac{\alpha_i}{\beta_i - \alpha_i}$.

Since we are assuming condition (2) does not hold, we have that $\frac{\alpha_i}{\alpha_i - \beta_i} \le \frac{\alpha_j}{\beta_j}$ for all $i, j$ with $i \neq j$. In particular, we have $\frac{\alpha_1}{\alpha_1 - \beta_1} \le \frac{\alpha_2}{\beta_2}$. Rearranging this gives $\frac{\beta_1}{\alpha_1} + \frac{\beta_2}{\alpha_2} \le 1$. Note that the two legs correspond to the fractions $-1/r := -\frac{\alpha_1}{\alpha_1 - \beta_1} = [a_1^1, \ldots, a_{h_1}^1]$ and $-1/s := -\frac{\alpha_2}{\alpha_2 - \beta_2} = [a_1^2, \ldots, a_{h_2}^2]$, where $r,s\in \Q$, and where our notation is as in Section \ref{sec:preliminaries}. Thus, we have that $r + s = 2 - \frac{\beta_1}{\alpha_1} - \frac{\beta_2}{\alpha_2} \ge 1$. Since $r + s \ge 1$, by Lemma \ref{lemma:shorten_legs} there exist $h_1' \le h_1$ and $h_2' \le h_2$ such that $-1/r_0 = [a_1^1, \ldots, a_{h_1'}^1]$ and $-1/s_0 = [a_1^2,\ldots, a_{h_2'}^2]$ with $r_0 + s_0 = 1$.

Let $\Psi$ be the union of the linear graph containing the first $h_1'$ vertices of the first leg (where we count vertices in a leg starting away from the central vertex), and the linear graph containing the first $h_2'$ vertices of the second leg. By restricting our embedding of $(\Z^k, Q_{\Gamma})$, we have an embedding of the sublattice corresponding to $\Psi$ into $(\Z^n, -\mbox{Id})$. The image of this embedding is contained in a sublattice $(\Z^d, -\mbox{Id})$ of $(\Z^n, -\mbox{Id})$ spanned by $\{e_i \in \Z^n\;|\; e_i \cdot v \neq 0\mbox{ for some vertex }v\mbox{ of }\Psi\}$. Hence $U_{\Psi}$ consists of $d$ elements (see Lemma \ref{lemma:rigidity} for definition of $U_{\Psi}$). Let $v_1, w_1$ be the two vertices of $\Psi$ adjacent to the central vertex in $\Gamma$. By our choice of the two legs of $\Gamma$ which contain the vertices of $\Psi$, we know that $e_j \in U_{v_1} \cap U_{w_1}$ for some $j \in \{1,\ldots,n\}$. This shows that the hypothesis of Lemma \ref{lemma:rigidity} are satisfied, hence we conclude that $d = h_1' + h_2'$.

Let $A$ be the matrix representing the embedding $(\Z^k, Q_{\Gamma})$ into $(\Z^n, -\mbox{Id})$. Then the $h_1' + h_2'$ columns of $A$ corresponding to the vertices of $\Psi$ are supported in $d = h_1' + h_2'$ rows of $A$ corresponding to the $d$-dimensional sublattice of $(\Z^n, -\mbox{Id})$. Denote this $d \times d$ minor by $B$. Then $-B^TB$ is a matrix for the intersection form of the plumbing corresponding to $\Psi$. Hence $-B^TB$ is a presentation matrix for $H_1(Y)$ where $Y$ is the boundary of the (disconnected) plumbing corresponding to $\Psi$. The $3$-manifold $Y$ is the disjoint union of two lens spaces, each given by surgery on the unknot with framings $-1/r_0 < -1$ and $-1/s_0 < -1$ respectively. Therefore $|\mbox{det}(B)|^2 = |H_1(Y)| > 1$ contradicting Lemma \ref{lemma:greene_obstruction}. Thus, $L$ is not quasi-alternating.
\end{proof}

\maincorol*
\begin{proof} This is a corollary of the proof of Theorem \ref{thm:necessary}. Suppose first that $L$ is quasi-alternating. By \cite[Proposition 3.3]{MR2141852}, $\Sigma(L)$ is an L-space. Furthermore, $\Sigma(L)$ must bound a negative definite $4$-manifold $W_1$ with $H_1(W_1) = 0$ \cite[Proof of Lemma 3.6]{MR2141852}. Applying this to the reflection of $L$ which is also quasi-alternating, we get that $\Sigma(L)$ also bounds a positive definite $4$-manifold $W_2$ with $H_1(W_2) = 0$.
  For the converse, note that these two necessary conditions are the only conditions used to obstruct a Montesinos link from being quasi-alternating in the proof of Theorem \ref{thm:necessary}.
\end{proof}

As a consequence, we obtain a classification of the Seifert fibered spaces which are formal L-spaces. Before stating it, we recall the definition of a formal L-space. We say that a triple $(Y_1, Y_2, Y_3)$ of closed, oriented $3$-manifolds form a \emph{triad} if there is a $3$-manifold $M$ with torus boundary, and three oriented curves $\gamma_1,\gamma_2,\gamma_3 \subset \partial M$ at pairwise distance $1$, such that $Y_i$ is the result of Dehn filling $M$ along $\gamma_i$, for $i=1,2,3$.

\begin{mydef} The set $\mathcal{F}$ of formal L-spaces is the smallest set of rational homology $3$-spheres such that
  \begin{enumerate}[(1)]
    \item $S^3 \in \mathcal{F}$, and
    \item if $(Y, Y_0, Y_1)$ is a triad with $Y_0, Y_1 \in \mathcal{F}$ and $$|H_1(Y)| = |H_1(Y_0)| + |H_1(Y_1)|,$$ then $Y \in \mathcal{F}$.
  \end{enumerate}
\end{mydef}

\formallspace*

\begin{proof} Let $L$ be a quasi-alternating Montesinos link. Then the double branched cover of $L$ is a Seifert fibered space over $S^2$. Ozsv\'{a}th and Szab\'{o} show that the double branched cover of a quasi-alternating link is an L-space \cite[Proposition 3.3]{MR2141852}. Their proof in fact shows that the double branched cover of a quasi-alternating link is a formal L-space. Hence $\Sigma(L)$ is a formal L-space Seifert fibered space over $S^2$.

  Now let $M$ be a formal L-space Seifert fibered space over $S^2$. Then $M$ is the double branched cover of a Montesinos link $L$. Ozsv\'{a}th and Szab\'{o}'s in \cite[Proof of Lemma 3.6]{MR2141852} show that the double branched cover of a quasi-alternating link bounds both a positive definite, and a negative definite $4$-manifold with vanishing first homology. However, their proof in fact shows this for all formal L-spaces. Hence $M = \Sigma(L)$ is a formal L-space bounding positive and negative definite $4$-manifolds with vanishing first homology. Thus, Corollary \ref{cor:qa} implies that $L$ is quasi-alternating.
\end{proof}

\section*{Acknowledgements}
I would like to thank Cameron Gordon for his support and helpful conversations, and Duncan McCoy for his suggestions and many helpful comments. I would also like to thank the referee for useful feedback.

\phantomsection
\bibliography{references}{}
\bibliographystyle{alpha}

\end{document}